\documentclass[1[leqno,11pt]{amsart}
\usepackage{amssymb, amsmath,amsmath,latexsym,amssymb,amsfonts,amsbsy, amsthm}

\let\pa=\partial
\let\al=\alpha

\let\f=\frac
\let\p=\psi

\let\ve=\varepsilon
\let\pa=\partial

\def\non{\nonumber}
\def\na{\nabla}

\def\p{\partial}

\def\eqdefa{\buildrel\hbox{\footnotesize def}\over =}

\newcommand{\beq}{\begin{equation}}
\newcommand{\eeq}{\end{equation}}
\newcommand{\ben}{\begin{eqnarray}}
\newcommand{\een}{\end{eqnarray}}
\newcommand{\beno}{\begin{eqnarray*}}
\newcommand{\eeno}{\end{eqnarray*}}

\renewcommand{\theequation}{\thesection.\arabic{equation}}

\newtheorem{Theorem}{Theorem}[section]

\newtheorem{Proposition}[Theorem]{Proposition}
\newtheorem{Lemma}[Theorem]{Lemma}

\newcommand{\ud}{\mathrm{d}}
\newcommand{\ue}{\mathrm{e}}

\newcommand{\vv}{\mathbf{v}}

\newcommand{\xx}{\mathbf{x}}

\newcommand{\nn}{\mathbf{n}}

\newcommand{\hh}{\mathbf{h}}
\newcommand{\mm}{\mathbf{m}}

\newcommand{\NN}{\mathbf{N}}

\newcommand{\DD}{\mathbf{D}}
\newcommand{\FF}{\mathbf{F}}
\newcommand{\II}{\mathbf{I}}

\newcommand{\CR}{\mathcal{R}}

\newcommand{\CU}{\mathcal{U}}

\newcommand{\CJ}{\mathcal{J}}

\newcommand{\BP}{\mathbb{P}}
\newcommand{\BS}{{\mathbb{S}^2}}
\newcommand{\BR}{{\mathbb{R}^3}}
\newcommand{\BN}{{\mathbb{N}^3}}
\newcommand{\BOm}{\mathbf{\Omega}}

\newcommand{\mue}{{\mu_{\varepsilon}}}

\newcommand{\sqe}{{\sqrt{\varepsilon}}}

\begin{document}
\title[The Ericksen-Leslie system]
{Well-posedness of the Ericksen-Leslie system}

\author{Wei Wang}
\address{School of  Mathematical Sciences, Peking University, Beijing 100871, China}
\email{wangw07@pku.edu.cn}

\author{Pingwen Zhang}
\address{School of  Mathematical Sciences, Peking University, Beijing 100871, China}
\email{pzhang@pku.edu.cn}

\author{Zhifei Zhang}
\address{School of  Mathematical Sciences, Peking University, Beijing 100871, China}
\email{zfzhang@math.pku.edu.cn}

\date{\today}

\maketitle
\begin{abstract}
In this paper, we  prove the local well-posedness of the Ericksen-Leslie system, and the global well-posednss for small
initial data under the physical constrain condition on the Leslie coefficients, which ensures that the
energy of the system is dissipated. Instead of the Ginzburg-Landau approximation,
we construct an approximate system with the dissipated energy based on a new formulation of the system.
\end{abstract}

\renewcommand{\theequation}{\thesection.\arabic{equation}}
\setcounter{equation}{0}
\section{Introduction}

The hydrodynamic theory of liquid crystals was established by  Ericksen \cite{E-61, E-91} and Leslie \cite{Les} in the 1960's.
This theory treats the liquid crystal material as a continuum and
completely ignores molecular details. Moreover, this theory considers perturbations to a presumed oriented sample.
The configuration of the liquid crystals is described by a director
field $\nn(t, \xx)\in \BS, \xx\in\BR$.

The general Ericksen-Leslie system takes the form
\begin{eqnarray}\label{eq:EL}
\left\{
\begin{split}
&\vv_t+\vv\cdot\nabla\vv=-\nabla{p}+\frac{\gamma}{Re}\Delta\vv
+\frac{1-\gamma}{Re}\nabla\cdot\sigma,\\
&\na\cdot\vv=0,\\
&\nn\times\big(\hh-\gamma_1\NN-\gamma_2\DD\cdot\nn\big)=0,
\end{split}\right.
\end{eqnarray}
where $\vv$ is the velocity of the fluid, $p$ is the pressure, $Re$ is  the Reynolds number and $\gamma\in (0,1)$.
The stress $\sigma$ is modeled by the phenomenological constitutive relation
\beno
\sigma=\sigma^L+\sigma^E,
\eeno
where $\sigma^L$ is the viscous (Leslie) stress
\begin{eqnarray}\label{eq:Leslie stress}
\sigma^L=\alpha_1(\nn\nn:\DD)\nn\nn+\alpha_2\nn\NN+\alpha_3\NN\nn+\alpha_4\DD
+\alpha_5\nn\nn\cdot\DD+\alpha_6\DD\cdot\nn\nn \een
with $\DD=\frac{1}{2}(\kappa^T+\kappa), \kappa=(\na \vv)^T$, and
\beno
\NN=\nn_t+\vv\cdot\nabla\nn+\BOm\cdot\nn,\quad\BOm=\frac{1}{2}(\kappa^T-\kappa).
\eeno
The six constants $\al_1, \cdots, \al_6$ are called the Leslie coefficients.  While, $\sigma^E$ is the elastic (Ericksen) stress
\begin{eqnarray}\label{eq:Ericksen}
\sigma^E=-\frac{\partial{E_F}}{\partial(\nabla\nn)}\cdot(\nabla\nn)^T,
\end{eqnarray}
where $E_F=E_F(\nn,\nabla\nn)$ is the Oseen-Frank energy with the form
\beno
E_F=\f {k_1} 2(\na\cdot\nn)^2+\f {k_2} 2|\nn\times(\na\times\nn)|^2+\f {k_3} 2|\nn\cdot(\na\times \nn)|^2.
\eeno
Here $k_1, k_2, k_3$ are the elastic constant.
For the simplicity, we will consider the case $k_1=k_2=k_3=1$.
In such case, $E_F=\f12|\na\nn|^2$, and the molecular field $\hh$ is given by
\beno
&&\hh=-\frac{\delta{E_F}}{\delta{\nn}}=
\nabla\cdot\frac{\partial{E_F}}{\partial(\nabla\nn)}-\frac{\partial{E_F}}{\partial\nn}=-\Delta\nn,\\
&&\big(\sigma^E\big)_{ij}=-\big(\nabla\nn\odot\nabla\nn\big)_{ij}=-\partial_in_k\partial_jn_k.
\eeno
Finally, the Leslie coefficients and $\gamma_1, \gamma_2$ satisfy the following relations
\ben
&\alpha_2+\alpha_3=\alpha_6-\alpha_5,\label{Leslie relation}\\
&\gamma_1=\alpha_3-\alpha_2,\quad \gamma_2=\alpha_6-\alpha_5,\label{Leslie-coeff}
\een
where (\ref{Leslie relation}) is called Parodi's relation derived from the Onsager reciprocal relation \cite{Parodi}. These two relations
ensure that the system has a basic energy law.

As the general Ericksen-Leslie system is very complicated, most of earlier works treated the simplified(or approximated) system of (\ref{eq:EL}).
Motivated by the work on the harmonic heat flow, Lin and Liu \cite{LL-ARMA} add the penality term $\f 1 {4\varepsilon^2}(|\nn|^2-1)^2$ in $W$
in order to remove some higher-order nonlinearities due to the constraint $|\nn|=1$. In such case, the system becomes
\begin{eqnarray}\label{eq:EL-pen}
\left\{
\begin{split}
&\vv_t+\vv\cdot\nabla\vv=-\nabla{p}+\frac{\gamma}{Re}\Delta\vv
+\frac{1-\gamma}{Re}\nabla\cdot\sigma,\\
&\nn_t+\vv\cdot\nabla\nn+\BOm\cdot\nn-\mu_1\Delta\nn-\mu_2\DD\cdot\nn-\f1 {\varepsilon^2}(|\nn|^2-1)\nn=0.
\end{split}\right.
\end{eqnarray}
This is so called the Ginzburg-Landau approximation.
They proved the global existence of weak solution and the local existence and uniqueness of strong solution
of the system (\ref{eq:EL-pen}) under certain strong constrains on the Leslie coefficients.
We refer to \cite{WXL} for a recent result about the role of Parodi's relation
in the well-posedness and stability. However, whether the solution of (\ref{eq:EL-pen}) converges to that
of (\ref{eq:EL}) as $\varepsilon$ tends to zero is still a challenging question. When neglecting the Leslie stress
$\sigma_L$ in (\ref{eq:EL}), a simplest system preserving the basic energy law is the following
\begin{eqnarray}\label{eq:EL-simple}
\left\{
\begin{split}
&\vv_t+\vv\cdot\nabla\vv-\Delta\vv+\nabla{p}=-\nabla\cdot\big(\na \nn\odot \na \nn\big),\\
&\nn_t+\vv\cdot\nabla\nn-\Delta \nn=|\na\nn|^2\nn.
\end{split}\right.
\end{eqnarray}
For this system, the local existence and uniqueness of strong solution can be proved by using the standard energy method;
see \cite{Wang} for the well-posedness result with rough data.
Huang and Wang \cite{HW} give the following BKM type blow-up criterion: Let $T^*$ be the maximal existence time of the strong solution.
If $T^*<\infty$, then it is necessary
\beno
\int_0^{T^*}\|\na \times \vv(t)\|_{L^\infty}+\|\na\nn(t)\|_{L^\infty}^2\ud t=+\infty.
\eeno
In two dimensional case, the global existence of weak solution has been independently proved by Lin, Lin and Wang \cite{LLW} and Hong \cite{Hong},
where they construct a class of weak solution with at most a finite number of singular times. The uniqueness of weak solution is proved by Lin-Wang \cite{LW}
and Xu-Zhang \cite{XZ}.
The global existence of weak solution of (\ref{eq:EL-simple}) is a challenging open problem in three dimensional case.
On the other hand, in the case when $|\na\nn|^2\nn$ in (\ref{eq:EL-simple}) is replaced by $\f 1 {\ve^2}(|\nn|^2-1)\nn$, the global existence and partial regularity
of weak solution  were studied in \cite{LL-CPAM, LL-DCDS}.

The purpose of this paper is to study the well-posedness of the general Ericksen-Leslie system.
The first step is to understand the complicated energy-dissipation law of the system arising from the Leslie stress.
 Moreover, whether the energy defined in (\ref{eq:energy law}) is dissipated remains unknown in physics,
 since the Leslie coefficients are difficult to determine by using experimental results.
 We present a sufficient and necessary condition on the Leslie coefficients to ensure that the energy of the system is dissipated.
The next step is to construct an approximate system with the dissipated energy
under the physical condition on the Leslie coefficients. However, the Ginzburg-Landau approximation does not satisfy our requirement.
We introduce a new equivalent formulation of the system (\ref{eq:EL}). Based on this formulation, we can construct an approximate system
such that the energy is still dissipated, although the key property $|\nn|=1$ is destroyed.

Our main results are stated as follows.
\begin{Theorem}\label{thm:local}
Let $s\ge 2$ be an integer. Assume that the Leslie coefficients satisfy (\ref{Leslie condition}),
and the initial data $\na\nn_0\in H^{2s}(\BR), \vv_0\in H^{2s}(\BR)$. There exist $T>0$ and a unique solution
$(\vv, \nn)$ of the Ericksen-Leslie system (\ref{eq:EL}) such that
\beno
\vv\in C([0,T]; H^{2s}(\BR))\cap L^2(0,T; H^{2s+1}(\BR)),\quad \na\nn\in C([0,T]; H^{2s}(\BR)).
\eeno
Let $T^*$ be the maximal existence  time of the solution. If $T^*<+\infty$, then it is necessary
\beno
\int_0^{T^*}\|\na \times \vv(t)\|_{L^\infty}+\|\na\nn(t)\|_{L^\infty}^2\ud t=+\infty.
\eeno

\end{Theorem}

For small initial data, we prove the following global well-posedness.

\begin{Theorem}\label{thm:global}
With the same assumptions as in Theorem \ref{thm:local}, there exists an $\varepsilon_0>0$ such that if
\beno
\|\na\nn_0\|_{H^{2s}}+\|\vv_0\|_{H^{2s}}\le \ve_0,
\eeno
then the solution obtained in Theorem \ref{thm:local} is global in time.
\end{Theorem}

The other sections of this  paper are organized as follows. In section 2, we derive the basic energy law of the system and give the physical constrain
condition on the Leslie coefficients. In section 3, we introduce a new equivalent formulation. Section 4 is devoted to the proof of local well-posedness.
In section 5, we prove the global well-posedness of the system for small initial data.

\setcounter{equation}{0}
\section{Basic energy-dissipation law}

We first derive the basic energy law of the system (\ref{eq:EL}).

\begin{Proposition}\label{prop:energy law}
If $(\vv,\nn)$ is a smooth solution of (\ref{eq:EL}), then it holds that
\begin{align}
\frac{\ud}{\ud{t}}\int_{\BR}\frac{Re}{2(1-\gamma)}|\vv|^2+E_F\ud\xx
=-&\int_{\BR}\Big(\frac{\gamma}{1-\gamma}|\nabla\vv|^2+(\alpha_1+\frac{\gamma_2^2}{\gamma_1})|\DD:\nn\nn|^2
+\alpha_4\DD:\DD\nonumber\\
&\quad+(\alpha_5+\alpha_6-\frac{\gamma_2^2}{\gamma_1})|\DD\cdot\nn|^2
+\frac{1}{\gamma_1}|\nn\times\hh|^2\Big)\ud\xx.\quad\label{eq:energy law}
\end{align}
\end{Proposition}

\begin{proof} Using the first equation of (\ref{eq:EL}) and $\na\cdot\vv=0$, we get
\begin{eqnarray}
&&\frac{\ud}{\ud{t}}\int_{\BR}\frac{Re}{2(1-\gamma)}|\vv|^2+E_F\ud\xx\nonumber\\
&&=\int_{\BR}\frac{Re}{1-\gamma}\vv\cdot\vv_t\ud\xx+\int_{\BR}
\frac{\delta{E_F}}{\delta\nn}\cdot \nn_t\ud\xx\nonumber\\
&&=-\int_{\BR}\frac{\gamma}{1-\gamma}|\nabla\vv|^2+(\sigma^L+\sigma^E):\nabla\vv\ud\xx+\int_{\BR}
\frac{\delta{E_F}}{\delta\nn}\cdot(\dot{\nn}-\vv\cdot\nabla\nn)\ud\xx,\label{eq:energ law-1}
\end{eqnarray}
where $\dot\nn=\nn_t+\vv\cdot\nabla\nn$. Using $\na\cdot\vv=0$ again, we have
\begin{eqnarray}
&&\int_{\BR}\sigma^E:\nabla\vv+\frac{\delta{E_F}}{\delta\nn}\cdot(\vv\cdot\nabla\nn)\ud\xx\nonumber\\
&&=\int_{\BR}(-\frac{\partial{E_F}}{\partial(\nabla\nn)}\cdot(\nabla\nn)^T):\nabla\vv-
\Big(\nabla\cdot\frac{\partial{E_F}}{\partial(\nabla\nn)}
-\frac{\partial{E_F}}{\partial\nn}\Big)\cdot(\vv\cdot\nabla\nn)\ud\xx\nonumber\\
&&=\int_{\BR}\frac{\partial{E_F}}{\partial(\nabla\nn)}:(\vv\cdot\nabla^2\nn)
+\frac{\partial{E_F}}{\partial\nn}\cdot(\vv\cdot\nabla\nn)\ud\xx\nonumber\\\label{eq:energ law-2}
&&=\int_{\BR}\vv\cdot\nabla{E_F}(\nn,\nabla\nn)\ud\xx=0.
\end{eqnarray}
Due to (\ref{eq:Leslie stress}), (\ref{Leslie relation}) and (\ref{Leslie-coeff}), we find
\begin{eqnarray}
&&\int_{\BR}\sigma^L:\nabla\vv\ud\xx\nonumber\\
&&=\int_{\BR}\Big( (\alpha_1(\nn\nn\cdot\DD)\nn\nn+\alpha_2\nn\NN+\alpha_3\NN\nn+\alpha_4\DD
+\alpha_5\nn\nn\cdot\DD+\alpha_6\DD\cdot\nn\nn):(\DD+\BOm)\Big)\ud\xx\nonumber\\
&&=\int_{\BR}\Big(\alpha_1(\nn\nn:\DD)^2+\alpha_4\DD:\DD+(\alpha_5+\alpha_6)|\DD\cdot\nn|^2
+(\alpha_2+\alpha_3)\nn\cdot(\DD\cdot\NN)\nonumber\\
&&\qquad\quad+(\alpha_2-\alpha_3)\nn\cdot(\BOm\cdot\NN)
-(\alpha_5-\alpha_6)(\DD\cdot\nn)\cdot(\BOm\cdot\nn)\Big)\ud\xx\nonumber\\
&&=\int_{\BR}\Big(\alpha_1(\nn\nn:\DD)^2+\alpha_4\DD:\DD+(\alpha_5+\alpha_6)|\DD\cdot\nn|^2
+\gamma_2\nn\cdot(\DD\cdot\NN)\nonumber\\
&&\qquad-\gamma_1\nn\cdot(\BOm\cdot\NN)
+\gamma_2(\DD\cdot\nn)\cdot(\BOm\cdot\nn)\Big)\ud\xx\nonumber\\
&&=\int_{\BR}\Big(\alpha_1(\nn\nn:\DD)^2+\alpha_4\DD:\DD+(\alpha_5+\alpha_6)|\DD\cdot\nn|^2
+\gamma_2\NN\cdot(\DD\cdot\nn)\nonumber\\
&&\qquad+\gamma_1\NN\cdot(\BOm\cdot\nn)
+\gamma_2(\DD\cdot\nn)\cdot(\BOm\cdot\nn)\Big)\ud\xx,\nonumber
\end{eqnarray}
and
\begin{eqnarray}
&&-\int_{\BR}\frac{\delta{E_F}}{\delta\nn}\cdot\dot{\nn}\ud\xx=\int_{\BR}\hh\cdot\dot{\nn}\ud\xx
=\int_{\BR}\hh\cdot(\NN-\BOm\cdot\nn)\ud\xx.\nonumber
\end{eqnarray}
The third equation of (\ref{eq:EL}) implies that
\beno
\int_{\BR}(\BOm\cdot\nn)\cdot\big(\gamma_1\NN
+\gamma_2(\DD\cdot\nn)-\hh\big)\ud\xx=0,
\eeno
and direct calculations show that
\begin{align}
\int_{\BR}\big(\gamma_2\NN\cdot(\DD\cdot\nn)+\hh\cdot\NN\big)\ud\xx
&=\int_{\BR}\big(\nn\times\NN\big)\cdot\big(\nn\times\hh+\gamma_2\nn\times\DD\cdot\nn\big)\ud\xx\nonumber\\
&=\int_{\BR}\frac{1}{\gamma_1}\big(\nn\times\hh-\gamma_2\nn\times\DD\cdot\nn\big)
\cdot\big(\nn\times\hh+\gamma_2\nn\times\DD\cdot\nn\big)\ud\xx\nonumber\\
&=\int_{\BR}\frac{1}{\gamma_1}|\nn\times\hh|^2
-\frac{\gamma_2^2}{\gamma_1}|\DD\cdot\nn|^2+\frac{\gamma_2^2}{\gamma_1}|\nn\cdot\DD\cdot\nn|^2\ud\xx.\nonumber
\end{align}
Thus, we have
\ben
&&\int_{\BR}\sigma^L:\nabla\vv-\frac{\delta{E_F}}{\delta\nn}\cdot\dot{\nn}\ud\xx\nonumber\\
&&=\int_{\BR}\Big((\alpha_1+\frac{\gamma_2^2}{\gamma_1})|\DD:\nn\nn|^2
+\alpha_4\DD:\DD+(\alpha_5+\alpha_6-\frac{\gamma_2^2}{\gamma_1})|\DD\cdot\nn|^2
+\frac{1}{\gamma_1}|\nn\times\hh|^2\Big)\ud\xx.\label{eq:energ law-3}
\een
Then the energy law (\ref{eq:energy law}) follows from (\ref{eq:energ law-1})-(\ref{eq:energ law-3}).
\end{proof}

The following proposition presents a sufficient and necessary condition on the
Leslie coefficients to ensure that the energy is dissipated; see also \cite{LL-ARMA} for the related discussions on the choice of the Leslie coefficients.
We denote
\begin{align}
\beta_1=\alpha_1+\frac{\gamma_2^2}{\gamma_1},\quad\beta_2=\alpha_4,
\quad\beta_3=\alpha_5+\alpha_6-\frac{\gamma_2^2}{\gamma_1}.\non
\end{align}

\begin{Proposition} \label{prop:Leslie condition}
The following dissipation relation holds
\begin{align}\label{eq:dissipation}
\beta_1(\nn\nn:\DD)^2+\beta_2\DD:\DD+\beta_3|\DD\cdot\nn|^2\ge0
\end{align}
for any symmetric trace free matrix $\DD$
and unit vector $\nn$, if and only if
\begin{align}\label{Leslie condition}
\beta_2\ge0,\quad2\beta_2+\beta_3\ge0,\quad\frac{3}{2}\beta_2+\beta_3+\beta_1\ge0.
\end{align}
\end{Proposition}

\begin{proof}
By the rotation invariance, we may assume $\nn=(0,0,1)^T$ and $\DD=(D_{ij})_{3\times3}$ with $D_{11}+D_{22}+D_{33}=0$.
It is easy to get
\begin{align*}
&\beta_1(\nn\nn:\DD)^2+\beta_2\DD:\DD+\beta_3|\DD\cdot\nn|^2\\
&=\beta_1D_{33}^2+\beta_2(D_{11}^2+D_{22}^2+D_{33}^2+2D_{12}^2+2D_{32}^2+2D_{31}^2)+\beta_3(D_{31}^2+D_{32}^2+D_{33}^2)\\
&=2\beta_2D_{12}^2+(2\beta_2+\beta_3)(D_{31}^2+D_{32}^2)
+\beta_2(D_{11}^2+D_{22}^2)+(\beta_1+\beta_2+\beta_3)D_{33}^2\\
&=2\beta_2D_{12}^2+(2\beta_2+\beta_3)(D_{31}^2+D_{32}^2)
+\beta_2(D_{11}^2+D_{22}^2)+(\beta_1+\beta_2+\beta_3)(D_{11}+D_{22})^2.
\end{align*}
The inequality holds
$$2\beta_2D_{12}^2+(2\beta_2+\beta_3)(D_{31}^2+D_{32}^2)\ge0$$
for all $D_{12},D_{31},$ and $D_{32}$, if and only if $\beta_2\ge0,$ and $2\beta_2+\beta_3\ge0$.

As $D_{11}^2+D_{22}^2\ge\frac{1}{2}(D_{11}+D_{22})^2$, the inequality holds
$$\beta_2(D_{11}^2+D_{22}^2)+(\beta_1+\beta_2+\beta_3)(D_{11}+D_{22})^2\ge0$$
for all $D_{11}$ and $D_{22}$, if and only if $\frac{3}{2}\beta_2+\beta_3+\beta_1\ge0$.
\end{proof}

In \cite{WZZ}, we show that if the Ericksen-Leslie system is derived from the Doi-Onsager equation, then
the energy (\ref{eq:energy law}) is indeed dissipated. Let us make it precise. The nondimensional Doi-Onsager equation takes as follows
\begin{eqnarray}\label{eq:Doi-Onsager}
\left\{
\begin{split}
&\frac{\pa{f^\ve}}{\pa{t}}+\vv^\ve\cdot\nabla{f^\ve}=\frac{1}{\ve}\CR\cdot(\CR{f^\ve}+f^\ve
\CR\CU_\ve f^\ve)
-\CR\cdot(\mm\times\kappa^\ve\cdot\mm{f^\ve}),\\
&\frac{\pa{\vv^\ve}}{\pa{t}}+\vv^\ve\cdot\nabla\vv^\ve=-\nabla{p^\ve}+\frac{\gamma}{Re}\Delta\vv^\ve
+\frac{1-\gamma}{2Re}\nabla\cdot\big(\DD^\ve:\langle\mm\mm\mm\mm\rangle_{f^\ve}\big)\\
&\qquad\qquad\qquad\qquad+\frac{1-\gamma}{\ve
Re}(\nabla\cdot\tau^e_\ve+\FF^e_\ve),\label{eq:LCP-nonL-v}
\end{split}\right.
\end{eqnarray}
where $\ve$ is the Deborah number, $\kappa^\ve=(\na v^\ve)^T,
\DD^\ve=\f12\big(\kappa^\ve+(\kappa^\ve)^T\big)$, and \beno
&&\tau^e_\ve=-\langle\mm\mm\times\CR\mue\rangle_{f^\ve},\quad
\FF^e_\ve=-\langle\nabla\mue\rangle_{f^\ve},\quad \mue=\ln f^\ve+\CU_\ve f,\\
&&\CU_\ve
f=\alpha\int_{\BR}\int_{\BS}|\mm\times\mm'|^2\frac{1}{\sqe^{3}}g\big(\frac{\xx-\xx'}{\sqe}\big)f(\xx',\mm',t)\ud\mm'\ud\xx'.
\eeno
When $\ve$ is small, the solution $(f^\ve,\vv^\ve)$ of the system (\ref{eq:Doi-Onsager}) has the expansion
\beno
&&f^\ve=f_0(\mm\cdot\nn)+\ve f_1+\cdots,\\
&&\vv^\ve=\vv_0+\ve\vv_1+\cdots, \eeno
where $(\vv_0,\nn)$ is determined by (\ref{eq:EL}) with the Leslie coefficients given by
\ben\label{Leslie cofficients}
\begin{split}
&\alpha_1=-\frac{S_4}{2},\quad\alpha_2=-\frac{1}{2}\big(1+\frac{1}{\lambda}\big)S_2,\quad\alpha_3=-\frac{1}{2}\big(1-\frac{1}{\lambda}\big)S_2,
\\
&\alpha_4=\frac{4}{15}-\frac{5}{21}S_2-\frac{1}{35}S_4,\quad
\alpha_5=\frac{1}{7}S_4+\frac{6}{7}S_2,\quad
\alpha_6=\frac{1}{7}S_4-\frac{1}{7}S_2.
\end{split}
\een
Here $S_2=\langle{P}_2(\mm\cdot\nn)\rangle_{h_{\eta_1,\nn}}$,
$S_4=\langle{P}_4(\mm\cdot\nn)\rangle_{h_{\eta_1,\nn}}$ with
$P_k(x)$ the $k$-th Legendre polynomial and
$$h_{\eta_1,\nn}(\mm)=\frac{\ue^{\eta_1(\mm\cdot\nn)^2}}{\int_\BS\ue^{\eta_1(\mm\cdot\nn)^2}\ud\mm}.$$
Here $\eta_1$ and  $\lambda$ are constants depending only on $\al$. When the Leslie coefficients are given by
(\ref{Leslie cofficients}), we show that the dissipation relation (\ref{eq:dissipation}) holds;
see \cite{WZZ, KD, EZ} for the details.

\setcounter{equation}{0}
\section{A New formulation of the Ericksen-Leslie system}

Set  $\mu_1=\frac{1}{\gamma_1}, \mu_2=-\frac{\gamma_2}{\gamma_1} $.
The third equation of (\ref{eq:EL}) is equivalent to
\begin{align}
\nn_t+\vv\cdot\nabla\nn+\BOm\cdot\nn-
(\II-\nn\nn)\cdot(\mu_1\hh+\mu_2\DD\cdot\nn)=0,\nonumber
\end{align}
which can be written as
\begin{align}\label{eq:director-new}
\nn_t+\vv\cdot\nabla\nn+
\nn\times\big((\BOm\cdot\nn-\mu_1\hh-\mu_2\DD\cdot\nn)\times\nn\big)=0.
\end{align}
Substituting them into (\ref{eq:Leslie stress}),  we get
\begin{align}\nonumber
\sigma^L=&\beta_1(\nn\nn:\DD)\nn\nn-\frac{1}{2}(1+\mu_2)\nn(\II-\nn\nn)\cdot\hh
+\frac{1}{2}(1-\mu_2)(\II-\nn\nn)\cdot\hh\nn\nonumber\\
&+\beta_2\DD+\frac{\beta_3}{2}(\nn\DD\cdot\nn+\DD\cdot\nn\nn)\nonumber\\
=&\beta_1(\nn\nn:\DD)\nn\nn-\frac{1}{2}(1+\mu_2)\nn\nn\times(\hh\times\nn)
+\frac{1}{2}(1-\mu_2)\nn\times(\hh\times\nn)\nn\nonumber\\
&+\beta_2\DD+\frac{\beta_3}{2}(\nn\DD\cdot\nn+\DD\cdot\nn\nn).\label{eq:Leslie stress-new}
\end{align}

With the new formulation (\ref{eq:director-new}) and (\ref{eq:Leslie stress-new}),  we can derive the same energy law (\ref{eq:energy law})  without using
the constrain $|\nn|=1$.  To see it, we need the following important cancelation relations.

\begin{Lemma}\label{lem:cancelation}
It holds that
\begin{align*}
&\big(-\frac{1}{2}\nn\nn\times(\hh\times\nn)
+\frac{1}{2}\nn\times(\hh\times\nn)\nn\big):(\DD+\BOm)
-\big((\BOm\cdot\nn)\times\nn\big)\cdot\big(\hh\times\nn\big)=0,\\
&-\big(\frac{1}{2}\nn\nn\times(\hh\times\nn)
+\frac{1}{2}\nn\times(\hh\times\nn)\nn\big):(\DD+\BOm)
+(\hh\times\nn)\cdot\big((\DD\cdot\nn)\times\nn\big)=0.
\end{align*}

\end{Lemma}

\begin{proof} Direct calculations show that
\begin{align*}
&\big(-\frac{1}{2}\nn\nn\times(\hh\times\nn)
+\frac{1}{2}\nn\times(\hh\times\nn)\nn\big):(\DD+\BOm)
-\big((\BOm\cdot\nn)\times\nn\big)\cdot\big(\hh\times\nn\big)\nonumber\\
&=\big(\nn\times(\hh\times\nn)\nn\big):\BOm
-\big((\BOm\cdot\nn)\times\nn\big)\cdot\big(\hh\times\nn\big)\nonumber\\
&=\big(\nn\times(\hh\times\nn)\big)\cdot(\BOm\cdot\nn)
-\big((\BOm\cdot\nn)\times\nn\big)\cdot\big(\hh\times\nn\big)=0,
\end{align*}
and
\begin{align*}
&-\big(\frac{1}{2}\nn\nn\times(\hh\times\nn)
+\frac{1}{2}\nn\times(\hh\times\nn)\nn\big):(\DD+\BOm)
+(\hh\times\nn)\cdot\big((\DD\cdot\nn)\times\nn\big)\nonumber\\
&=-\big(\nn\times(\hh\times\nn)\nn\big):\DD
+\big((\DD\cdot\nn)\times\nn\big)\cdot(\hh\times\nn)\nonumber\\
&=-\big(\nn\times(\hh\times\nn)\big)\cdot(\DD\cdot\nn)
+\big((\DD\cdot\nn)\times\nn\big)\cdot(\hh\times\nn)=0.
\end{align*}
The proof is finished.
\end{proof}

Now we derive the energy law (\ref{eq:energy law})  by using (\ref{eq:director-new}) and (\ref{eq:Leslie stress-new}), since the derivation will be helpful
to understand the energy estimates in the next section.  Thanks to  (\ref{eq:EL}) and (\ref{eq:director-new}),  we have
\begin{align*}
&-\frac{1}{2}\frac{\ud}{\ud{t}}\int_{\BR}\frac{Re}{1-\gamma}|\vv|^2+|\nabla\nn|^2\ud\xx
=-\int_{\BR}\frac{Re}{1-\gamma}\vv\cdot\vv_t-\Delta\nn\cdot\nn_t\ud\xx\\
&=\int_{\BR}\frac{\gamma}{1-\gamma}|\nabla\vv|^2+(\sigma^L+\sigma^E):\nabla\vv-(\vv\cdot\nabla\nn)\cdot\hh\\
&\qquad+\big(\nn\times\big((\mu_1\hh+\mu_2\DD\cdot\nn-\BOm\cdot\nn)\times\nn\big)\big)\cdot\hh\ud\xx\\
&=\int_{\BR}\frac{\gamma}{1-\gamma}|\nabla\vv|^2+\mu_1|\hh\times\nn|^2+\sigma^E:\nabla\vv-(\vv\cdot\nabla\nn)\cdot\hh\\
&\qquad+\sigma^L:\nabla\vv-\big((\BOm\cdot\nn)\times\nn\big)\cdot\big(\hh\times\nn\big)
+\mu_2(\hh\times\nn)\cdot\big((\DD\cdot\nn)\times\nn\big)\ud\xx.
\end{align*}
For the Ericksen stress term, we have
\begin{align*}
&\int_{\BR}\sigma^E:\nabla\vv-(\vv\cdot\nabla\nn)\cdot\hh\ud\xx
=\int_\BR-\partial_in_k\partial_jn_k\partial_iv_j-v_j\partial_jn_k\partial_{ii}n_k\ud\xx\\
&\quad=\int_\BR{v}_j\partial_j\partial_in_k\partial_in_k-\partial_i(v_j\partial_jn_k\partial_{i}n_k)\ud\xx=0,
\end{align*}
while for the Leslie stress term,  we  get by (\ref{eq:Leslie stress-new}) and Lemma \ref{lem:cancelation}  that
\begin{align*}
&\int_{\BR}\sigma^L:\nabla\vv-\big((\BOm\cdot\nn)\times\nn\big)\cdot\big(\hh\times\nn\big)
+\mu_2(\hh\times\nn)\cdot\big((\DD\cdot\nn)\times\nn\big)\ud\xx\\
&=\int_{\BR}\Big(\beta_1(\nn\nn:\DD)\nn\nn+\frac{1}{2}(-1-\mu_2)\nn\nn\times(\hh\times\nn)
+\frac{1}{2}(1-\mu_2)\nn\times(\hh\times\nn)\nn+\beta_2\DD\\
&\qquad+\frac{\beta_3}{2}(\nn\DD\cdot\nn+\DD\cdot\nn\nn)\Big):(\DD+\BOm)-\big((\BOm\cdot\nn)\times\nn\big)\cdot\big(\hh\times\nn\big)
+\mu_2(\hh\times\nn)\cdot\big((\DD\cdot\nn)\times\nn\big)\ud\xx\\
&=\int_{\BR}\beta_1(\nn\nn:\DD)^2+\beta_2\DD:\DD+\beta_3|\DD\cdot\nn|^2-\big((\BOm\cdot\nn)\times\nn\big)\cdot\big(\hh\times\nn\big)
+\mu_2(\hh\times\nn)\cdot\big((\DD\cdot\nn)\times\nn\big)\\
&\qquad+\Big(\frac{1}{2}(-1-\mu_2)\nn\nn\times(\hh\times\nn)
+\frac{1}{2}(1-\mu_2)\nn\times(\hh\times\nn)\nn\Big):(\DD+\BOm)\ud\xx\\
&=\int_{\BR}\beta_1(\nn\nn:\DD)^2+\beta_2\DD:\DD+\beta_3|\DD\cdot\nn|^2\ud\xx.
\end{align*}
Then the energy law (\ref{eq:energy law})  follows from the above identities.

Although the energy law can be derived without using the property $|\nn|=1$, this property is vital for the dissipation relation (\ref{eq:dissipation})
under the condition (\ref{Leslie condition}).
Hence, it is important to construct an approximate system preserving  the energy law and $|\nn|=1$
in order to prove the local well-posedness of (\ref{eq:EL}). It is usually difficult. For this, we introduce a modified
stress tensor so that the energy is still dissipated for the modified system under the condition (\ref{Leslie condition}).
The modified Leslie stress tensor takes the form
\begin{align}\nonumber
\widetilde\sigma^L=&\beta_1(\nn\nn:\DD)\nn\nn+\frac{1}{2}(-1-\mu_2)\nn\nn\times(\hh\times\nn)
+\frac{1}{2}(1-\mu_2)\nn\times(\hh\times\nn)\nn\\
&+\beta_2|\nn|^4\DD+\frac{\beta_3}{2}|\nn|^2(\nn\DD\cdot\nn+\DD\cdot\nn\nn).\non
\end{align}
It is obvious that $\widetilde\sigma^L=\sigma^L$ if $|\nn|=1$.
An important fact is that  for any traceless symmetric $\DD$ and vector $\nn$ (not necessary unit),
it still holds
\begin{align}\label{dissipate relation}
\big\langle\beta_1(\nn\nn:\DD)\nn\nn+\beta_2|\nn|^4\DD+\frac{\beta_3}{2}|\nn|^2
(\nn\DD\cdot\nn+\DD\cdot\nn\nn),\DD\big\rangle\ge0
\end{align}
under the condition (\ref{Leslie condition}). We denote
\begin{align*}
&\sigma_1(\vv,\nn)=\beta_1(\nn\nn:\DD)\nn\nn+\beta_2|\nn|^4\DD+\frac{\beta_3}{2}|\nn|^2
(\nn\DD\cdot\nn+\DD\cdot\nn\nn),\\
&\sigma_2(\nn)=\frac{1}{2}(-1-\mu_2)\nn\nn\times(\hh\times\nn)
+\frac{1}{2}(1-\mu_2)\nn\times(\hh\times\nn)\nn.
\end{align*}
The reformulated new system takes
\ben\label{eq:EL-new}
\left\{
\begin{split}
&\vv_t+\vv\cdot\nabla\vv=-\nabla{p}+\nu\Delta\vv
+\nabla\cdot\big(\sigma_1(\vv,\nn)+\sigma_2(\nn)+\sigma^E\big),\\
&\nn_t+\vv\cdot\nabla\nn+
\nn\times\big((\BOm\cdot\nn-\mu_1\hh-\mu_2\DD\cdot\nn)\times\nn\big)=0.
\end{split}\right.
\een
Here we set $\nu=\frac \gamma {Re}$  and take $\frac {1-\gamma} {Re}=1$.
Similar to Proposition \ref{prop:energy law}, we can show that the system (\ref{eq:EL-new})  obeys the following energy-dissipation law:
\begin{align}
\frac{1}{2}\frac{\ud}{\ud{t}}\int_{\BR}|\vv|^2+|\nabla\nn|^2\ud\xx
=-&\int_{\BR}\Big(\nu|\nabla\vv|^2+\beta_1|\DD:\nn\nn|^2
+\beta_2|\nn|^4\DD:\DD\nonumber\\
&\quad+\beta_3|\nn|^2|\DD\cdot\nn|^2
+\mu_1|\nn\times\hh|^2\Big)\ud\xx,
\end{align}
which is dissipated under the condition (\ref{Leslie condition})  by (\ref{dissipate relation}).

\section{Local well-posedness and blow-up criterion}
This section is devoted to proving the local well-posedness of the system (\ref{eq:EL}).
The following lemma will frequently used.
\begin{Lemma}\label{lem:product}
For any $\al,\beta\in \BN$, it hods that
\begin{align}
&\|D^\al(fg)\|_{L^2}\le C\sum_{|\gamma|=|\al|}\big(\|f\|_{L^\infty}\|D^\gamma g\|_{L^2}+\|g\|_{L^\infty}\|D^\gamma f\|_{L^2}\big),\non\\
&\big\|[D^\al,f]D^\beta g\big\|_{L^2}
\le C\Big(\sum_{|\gamma|=|\al|+|\beta|}\|D^{\gamma}f\|_{L^2}\|g\|_{L^\infty}
+\sum_{|\gamma|=|\al|+|\beta|-1}\|\na f\|_{L^\infty}\|D^{\gamma}g\|_{L^2}\Big).\non
\end{align}
\end{Lemma}
This lemma can be easily proved by using Bony's decomposition; see \cite{Bah} for example.
The proof of Theorem \ref{thm:local} is split into several steps.\vspace{0.1cm}

{\bf Step 1.}\, Construction of the approximate solutions

The construction is based on the classical Friedrich's method.
Define the smoothing operator
\begin{align}
\CJ_\ve{f}=\mathcal{F}^{-1}(\mathbf{1}_{|\xi|{\le}\frac{1}{\ve}}\mathcal{F}f),
\end{align}
where $\mathcal{F}$ is the usual Fourier transform.
Let $\BP$ be the operator which projects a vector field to its solenoidal part.
We introduce the following approximate system of (\ref{eq:EL-new}):
\beno
\left\{
\begin{split}
&\frac{\pa{\vv_\ve}}{\pa{t}}+\CJ_\ve\mathbb{P}(\CJ_\ve\vv_\ve\cdot\nabla\CJ_\ve\vv_\ve)=
\nu\Delta\CJ_\ve\vv_\ve+\nabla\cdot\CJ_\ve\mathbb{P}\big(\sigma_1(\CJ_\ve\vv_\ve,\CJ_\ve\nn_\ve)
+\sigma_2(\CJ_\ve\nn_\ve)+\sigma^E(\CJ_\ve\nn_\ve)\big),\\\label{EL-final-n-approx}
&\frac{\pa\nn_\ve}{\pa{t}}+\CJ_\ve\Big(\CJ_\ve\vv_\ve\cdot\nabla\CJ_\ve\nn_\ve+
\CJ_\ve\nn_\ve\times\big[(\CJ_\ve\BOm_\ve\cdot\CJ_\ve\nn_\ve-\mu_1\CJ_\ve\hh_\ve
-\mu_2\CJ_\ve\DD_\ve\cdot\CJ_\ve\nn_\ve)\times\CJ_\ve\nn_\ve\big]\Big)=0,\\
&(\vv_\ve,\nn_\ve)|_{t=0}=(\CJ_\ve\vv_0,\CJ_\ve\nn_0).
\end{split}\right.
\eeno
The above system can be viewed as an ODE system on $L^2(\BR)$. Then we know by the Cauchy-Lipschitz
theorem that there exist a strictly maximal time $T_\ve$ and a unique
solution $(\vv_\ve,\nn_\ve)$ which is continuous in time with value in $H^k(\BR)$ for any $k\ge 0$.
As $\CJ^2_\ve=\CJ_\ve$, we know that $(\CJ_\ve\vv_\ve, \CJ_\ve\nn_\ve)$ is
also a solution. Therefore, $(\vv_\ve,\nn_\ve)=(\CJ_\ve\vv_\ve, \CJ_\ve\nn_\ve)$.
Thus,  $(\vv_\ve,\nn_\ve)$ satisfies the following system
\ben\label{eq:EL-app}
\left\{
\begin{split}
&\frac{\pa{\vv_\ve}}{\pa{t}}+\CJ_\ve\mathbb{P}(\vv_\ve\cdot\nabla\vv_\ve)=
\nu\Delta\vv_\ve+\nabla\cdot\CJ_\ve\mathbb{P}\big(\sigma_1(\vv_\ve,\nn_\ve)
+\sigma_2(\nn_\ve)+\sigma^E(\nn_\ve)\big),\\
&\frac{\pa\nn_\ve}{\pa{t}}+\CJ_\ve\Big(\vv_\ve\cdot\nabla\nn_\ve+
\nn_\ve\times\big[(\BOm_\ve\cdot\nn_\ve-\mu_1\hh_\ve
-\mu_2\DD_\ve\cdot\nn_\ve)\times\nn_\ve\big]\Big)=0,\\
&(\vv_\ve,\nn_\ve)|_{t=0}=(\CJ_\ve\vv_0,\CJ_\ve\nn_0).
\end{split}\right.
\een

{\bf Step 2.} Uniform energy estimates

We define
\begin{align}
E_s(\vv,\nn)\eqdefa\|\nn-\nn_0\|_{L^2}^2+\|\nabla\nn\|_{L^2}^2+\|\nabla\Delta^s\nn\|_{L^2}^2
+\|\vv\|_{L^2}^2+\|\Delta^s\vv\|_{L^2}^2.\non
\end{align}
First of all, we get by the second equation of (\ref{eq:EL-app})  that
\begin{align}\label{eq:n-L2}
&\frac{\ud}{\ud{t}}\|\nn_\ve-\nn_0\|_{L^2}^2=2\big\langle
\p_t\nn_\ve,\nn_\ve-\nn_0\big\rangle\nonumber\\
&=\big\langle\vv_\ve\cdot\nabla\nn_\ve+\nn_\ve\times
\big[(\BOm_\ve\cdot\nn_\ve-\mu_1\hh_\ve
-\mu_2\DD_\ve\cdot\nn_\ve)\times\nn_\ve\big],\CJ_\ve(\nn_\ve-\nn_0)\big\rangle\nonumber\\
&=\big\langle \vv_\ve\cdot\nabla\nn_0+\nn_\ve\times
\big[(\BOm_\ve\cdot\nn_\ve-\mu_1\hh_\ve-\mu_2\DD_\ve\cdot\nn_\ve)\times\nn_\ve\big],\CJ_\ve(\nn_\ve-\nn_0)\big\rangle\non\\
&\le C\big[\|\na \nn_0\|_{L^\infty}\|\vv_\ve\|_{L^2}+\|\nn_\ve\|_{L^\infty}^2\big(\|\nn_\ve\|_{L^\infty}\|\na\vv_\ve\|_{L^2}+\|\Delta\nn_\ve\|_{L^2}\big)\big]\|\nn_\ve-\nn_0\|_{L^2}\non\\
&\le C\big(\|\na \nn_0\|_{L^\infty}+\|\nn_\ve\|_{L^\infty}^2+\|\nn_\ve\|_{L^\infty}^3\big)E_s(\vv_\ve,\nn_\ve).
\end{align}
The following energy law still holds for the approximate system (\ref{eq:EL-app}):
\begin{align}\label{eq:energy law-app}
\frac{1}{2}\frac{\ud}{\ud{t}}\int_{\BR}|\vv_\ve|^2+|\nabla\nn_\ve|^2\ud\xx
=&-\int_{\BR}\Big(\nu|\nabla\vv_\ve|^2+\beta_1|\DD_\ve:\nn_\ve\nn_\ve|^2
+\beta_2|\nn_\ve|^4\DD_\ve:\DD_\ve\nonumber\\
&\qquad+\beta_3|\nn_\ve|^2|\DD_\ve\cdot\nn_\ve|^2+\mu_1|\nn_\ve\times\hh_\ve|^2\Big)\ud\xx.
\end{align}

Now we turn to the estimate of the higher order derivative for $\nn_\ve$.
\begin{align}\label{eq:energy-n}
\frac{1}{2}\frac{\ud}{\ud{t}}\big\langle\nabla\Delta^s\nn_\ve, \nabla\Delta^s\nn_\ve\big\rangle
&=\big\langle\Delta^s(\vv_\ve\cdot\nabla\nn_\ve),\Delta^{s+1}\nn_\ve\big\rangle
+\big\langle\Delta^s\big[\nn_\ve\times\big((\BOm_\ve\cdot\nn_\ve)\times\nn_\ve\big)\big],\Delta^{s+1}\nn_\ve\big\rangle\nonumber\\
&\quad-\mu_2\big\langle\Delta^s\big[\nn_\ve\times\big((\DD_\ve\cdot\nn_\ve)\times\nn_\ve\big)\big],\Delta^{s+1}\nn_\ve\big\rangle
-\mu_1\big\langle\Delta^s\big[\nn_\ve\times\big(\Delta\nn_\ve\times\nn_\ve\big)\big],\Delta^{s+1}\nn_\ve\big\rangle\nonumber\\
&=I_1+I_2+I_3+I_4.
\end{align}
As $\na \cdot \vv_\ve=0$, we  get by Lemma \ref{lem:product} that
\begin{align}\label{eq:energy-n1}
I_1&=-\big\langle\nabla\Delta^s(\vv_\ve\cdot\nabla\nn_\ve),\nabla\Delta^{s}\nn_\ve\big\rangle
+\big\langle\vv_\ve\cdot\nabla(\nabla\Delta^s\nn_\ve),\nabla\Delta^{s}\nn_\ve\big\rangle\non\\
&=-\big\langle[\nabla\Delta^s, \vv_\ve]\cdot\nabla\nn_\ve,\nabla\Delta^{s}\nn_\ve\big\rangle\non\\
&\le \|[\nabla\Delta^s, \vv_\ve]\cdot\nabla\nn_\ve\|_{L^2}\|\nabla\Delta^{s}\nn_\ve\|_{L^2}\non\\
&\le C\big(\|\na \nn_\ve\|_{H^{2s}}\|\na\vv_\ve\|_{L^\infty}+\|\na \vv_\ve\|_{H^{2s}}\|\na \nn_\ve\|_{L^\infty}\big) \|\nabla\Delta^{s}\nn_\ve\|_{L^2}\non\\
&\le C_\delta\big(\|\na\vv_\ve\|_{L^\infty}+\|\na \nn_\ve\|_{L^\infty}^2\big) \|\na \nn_\ve\|_{H^{2s}}^2+\delta\|\na \vv_\ve\|_{H^{2s}}^2.
\end{align}
Here and in what follows,  $\delta$ denotes a positive constant to be determined later.  We rewrite $I_2$ as
\begin{align}
I_2=&\big\langle\nn_\ve\times\big((\Delta^s\BOm_\ve\cdot\nn_\ve)\times\nn_\ve\big),\Delta^{s+1}\nn_\ve\big\rangle\non\\
&-\big\langle\na\Delta^s\big[\nn_\ve\times\big((\BOm_\ve\cdot\nn_\ve)\times
\nn_\ve\big)\big], \na\Delta^{s}\nn_\ve\big\rangle+\big\langle\nn_\ve\times\big((\na\Delta^s\BOm_\ve\cdot\nn_\ve)\times\nn_\ve\big), \na\Delta^{s}\nn_\ve\big\rangle\non\\
&+\big\langle (\na\nn_\ve)\times\big((\Delta^s\BOm_\ve\cdot\nn_\ve)\times\nn_\ve\big), \na\Delta^{s}\nn_\ve\big\rangle
+\big\langle\nn_\ve\times\big((\Delta^s\BOm_\ve\cdot(\na\nn_\ve))\times\nn_\ve\big), \na\Delta^{s}\nn_\ve\big\rangle\non\\
&+\big\langle\nn_\ve\times\big((\Delta^s\BOm_\ve\cdot\nn_\ve)\times(\na\nn_\ve)\big),\na\Delta^{s}\nn_\ve\big\rangle,\non
\end{align}
from which and Lemma  \ref{lem:product}, it follows that
\begin{align}\label{eq:energy-n2}
I_2\le& \big\langle\nn_\ve\times\big((\Delta^s\BOm_\ve\cdot\nn_\ve)\times\nn_\ve\big),\Delta^{s+1}\nn_\ve\big\rangle\non\\
&+C_\delta\big(\|\nn_\ve\|_{L^\infty}^2\|\na\vv_\ve\|_{L^\infty}+\|\nn_\ve\|_{L^\infty}^4\|\na\nn_\ve\|_{L^\infty}^2\big)\|\na \nn_\ve\|_{H^{2s}}^2
+\delta\|\na \vv_\ve\|_{H^{2s}}^2.
\end{align}
Similarly, we have
\begin{align}\label{eq:energy-n3}
I_3\le& -\mu_2\big\langle\nn_\ve\times\big((\Delta^s\DD_\ve\cdot\nn_\ve)\times\nn_\ve\big),\Delta^{s+1}\nn_\ve\big\rangle\non\\
&+C_\delta\big(\|\nn_\ve\|_{L^\infty}^2\|\na\vv_\ve\|_{L^\infty}+\|\nn_\ve\|_{L^\infty}^4\|\na\nn_\ve\|_{L^\infty}^2\big)\|\na \nn_\ve\|_{H^{2s}}^2
+\delta\|\na \vv_\ve\|_{H^{2s}}^2.
\end{align}
For $I_4$, we have
\begin{align}
I_4=&-\mu_1\big\langle \nn_\ve\times\big(\Delta^{s+1}\nn_\ve\times\nn_\ve\big), \Delta^{s+1}\nn_\ve\big\rangle\non\\
&-\big[\mu_1\big\langle\Delta^s\big[\nn_\ve\times\big(\Delta\nn_\ve\times\nn_\ve\big)\big],\Delta^{s+1}\nn_\ve\big\rangle
-\mu_1\big\langle\nn_\ve\times\Delta^s\big(\Delta\nn_\ve\times\nn_\ve\big),\Delta^{s+1}\nn_\ve\big\rangle\big]\non\\
&-\big[\mu_1\big\langle\nn_\ve\times\Delta^s\big(\Delta\nn_\ve\times\nn_\ve\big),\Delta^{s+1}\nn_\ve\big\rangle
-\mu_1\big\langle \nn_\ve\times\big(\Delta^{s+1}\nn_\ve\times\nn_\ve\big), \Delta^{s+1}\nn_\ve\big\rangle\big]\non\\
=&\mu_1\big\langle\Delta^{s+1}\nn_\ve\times\nn_\ve, \Delta^{s+1}\nn_\ve\times\nn_\ve\big\rangle+I_{41}+I_{42}.\non
\end{align}
We get by Lemma \ref{lem:product} that
\begin{align}
I_{42}&\le C\|\na\nn_\ve\|_{L^\infty}\|\na\nn_\ve\|_{H^{2s}}\|\Delta^{s+1}\nn_\ve\times\nn_\ve\|_{L^2}\non\\
&\le  C_\delta\|\na\nn_\ve\|_{L^\infty}^2\|\na\nn_\ve\|_{H^{2s}}^2+\delta\|\Delta^{s+1}\nn_\ve\times\nn_\ve\|_{L^2}^2,\non
\end{align}
and
\begin{align}
I_{41}=&\mu_1\big[\big\langle\Delta^s\big[\na\nn_\ve\times\big(\Delta\nn_\ve\times\nn_\ve\big)\big],\na\Delta^{s}\nn_\ve\big\rangle
-\big\langle(\na\nn_\ve)\times\Delta^s\big(\Delta\nn_\ve\times\nn_\ve\big),\na\Delta^{s}\nn_\ve\big\rangle\big]\non\\
&\mu_1\big[\big\langle\Delta^s\big[\nn_\ve\times\na\big(\Delta\nn_\ve\times\nn_\ve\big)\big],\na\Delta^{s}\nn_\ve\big\rangle
-\big\langle\nn_\ve\times\na\Delta^s\big(\Delta\nn_\ve\times\nn_\ve\big),\na\Delta^{s}\nn_\ve\big\rangle\big]\non\\
\le &C\big(\|\na\nn_\ve\|_{L^\infty}\|\Delta^s\big(\Delta\nn_\ve\times\nn_\ve\big)\|_{L^2}+
\|\na\nn_\ve\|_{H^{2s}}\|\Delta\nn_\ve\times\nn_\ve\|_{L^\infty}\big)\|\na\nn_\ve\|_{H^{2s}}\non\\
\le &C\big\{\|\na\nn_\ve\|_{L^\infty}(\|\Delta^{s+1}\nn_\ve\times\nn_\ve\|_{L^2}+\|\na\nn_\ve\|_{L^\infty}\|\na\nn_\ve\|_{H^{2s}})\non\\
&\qquad\quad+\|\na\nn_\ve\|_{H^{2s}}\|\Delta\nn_\ve\|_{L^\infty}\|\nn_\ve\|_{L^\infty}\big\}\|\na\nn_\ve\|_{H^{2s}},\non
\end{align}
which imply that
\begin{align}\label{eq:energy-n4}
I_4\le& -\mu_1\big\langle\Delta^{s+1}\nn_\ve\times\nn_\ve, \Delta^{s+1}\nn_\ve\times\nn_\ve\big\rangle\non\\
&+C_\delta\big(\|\na\nn_\ve\|_{L^\infty}^2+\|\Delta\nn_\ve\|_{L^\infty}\|\nn_\ve\|_{L^\infty}\big)\|\na\nn_\ve\|_{H^{2s}}^2+\delta\|\Delta^{s+1}\nn_\ve\times\nn_\ve\|_{L^2}^2.
\end{align}
Substituting (\ref{eq:energy-n1})-(\ref{eq:energy-n4}) into (\ref{eq:energy-n}), we infer that
\begin{align}\label{eq:energy-nf}
&\frac{1}{2}\frac{\ud}{\ud{t}}\big\langle\nabla\Delta^s\nn_\ve, \nabla\Delta^s\nn_\ve\big\rangle
+\mu_1\big\langle\Delta^{s+1}\nn_\ve\times\nn_\ve, \Delta^{s+1}\nn_\ve\times\nn_\ve\big\rangle\non\\
&\le\big\langle\nn_\ve\times\big((\Delta^s\BOm_\ve\cdot\nn_\ve)\times\nn_\ve\big),\Delta^{s+1}\nn_\ve\big\rangle
-\mu_2\big\langle\nn_\ve\times\big((\Delta^s\DD_\ve\cdot\nn_\ve)\times\nn_\ve\big),\Delta^{s+1}\nn_\ve\big\rangle\non\\
&\quad+C_\delta\big(\|\na\vv_\ve\|_{L^\infty}+\|\na \nn_\ve\|_{L^\infty}^2+\|\Delta\nn_\ve\|_{L^\infty}\|\nn_\ve\|_{L^\infty}\big)\|\na \nn_\ve\|_{H^{2s}}^2\non\\
&\quad+\delta\big(\|\na \vv_\ve\|_{H^{2s}}^2+\|\Delta^{s+1}\nn_\ve\times\nn_\ve\|_{L^2}^2\big).
\end{align}

Next we consider the estimate of the higher order derivative for $\vv_\ve$.
\begin{align}
&\frac{1}{2}\frac{\ud}{\ud{t}}\langle\Delta^s\vv_\ve,\Delta^s\vv_\ve\rangle+\nu\langle\nabla\Delta^s\vv_\ve,\nabla\Delta^s\vv_\ve\rangle\non\\
&=-\big\langle\Delta^s(\vv_\ve\cdot\nabla\vv_\ve),\Delta^s\vv_\ve\big\rangle
+\big\langle\Delta^s(\nabla\nn_\ve\odot\nabla\nn_\ve), \Delta^s\nabla\vv_\ve\big\rangle\non\\
&\quad-\Big\langle\Delta^s\Big(\beta_1(\nn_\ve\nn_\ve:\DD_\ve)\nn_\ve\nn_\ve+\beta_2|\nn_\ve|^4\DD_\ve
+\frac{\beta_3}{2}|\nn_\ve|^2(\nn_\ve\DD_\ve\cdot\nn_\ve+\DD_\ve\cdot\nn_\ve\nn_\ve)\non\\
&\qquad-\frac{1}{2}(1+\mu_2)\nn_\ve\nn_\ve\times(\hh_\ve\times\nn_\ve)
+\frac{1}{2}(1-\mu_2)\nn_\ve\times(\hh_\ve\times\nn_\ve)\nn_\ve\Big),\Delta^s\nabla\vv_\ve\Big\rangle\non\\
&=-\langle\Delta^s(\vv_\ve\cdot\nabla\vv_\ve),\Delta^s\vv_\ve\rangle
+\langle\Delta^s(\nabla\nn_\ve\odot\nabla\nn_\ve), \Delta^s\nabla\vv_\ve\rangle\non\\
&\quad-\Big\langle\Delta^s\big(\beta_1(\nn_\ve\nn_\ve:\DD_\ve)\nn_\ve\nn_\ve+\beta_2|\nn_\ve|^4\DD_\ve
+\frac{\beta_3}{2}|\nn_\ve|^2(\nn_\ve\DD_\ve\cdot\nn_\ve+\DD_\ve\cdot\nn_\ve\nn_\ve)\big),\Delta^s\DD_\ve\Big\rangle\non\\
&\qquad+\mu_2\big\langle\Delta^s\big(\nn_\ve\times(\hh_\ve\times\nn_\ve)\nn_\ve\big),\Delta^s\DD_\ve\big\rangle
-\big\langle\Delta^s(\nn_\ve\times(\hh_\ve\times\nn_\ve)\nn_\ve),\Delta^s\BOm_\ve\big\rangle\non\\
&=II_1+II_2+II_3+II_4+II_5.\non
\end{align}
It follows from Lemma \ref{lem:product} that
\begin{align}
&II_1=\big\langle[\Delta^s,\vv_\ve]\cdot\nabla\vv_\ve,\Delta^s\vv_\ve\big\rangle
\le C\|\na\vv_\ve\|_{L^\infty}\|\vv_\ve\|_{H^{2s}}^2,\non\\
&II_2\le C_\delta\|\na\nn_\ve\|_{L^\infty}^2\|\na\nn_\ve\|_{H^{2s}}^2+\delta\|\na\vv_\ve\|_{H^{2s}}^2,\non
\end{align}
and
\begin{align}
II_4\le&\mu_2\big\langle\nn_\ve\times(\Delta^{s+1}\nn_\ve\times\nn_\ve),\Delta^s\DD_\ve\cdot\nn_\ve\big\rangle\non\\
&+C_\delta\|\nabla\nn_\ve\|_{L^\infty}^2\|\nn_\ve\|_{L^\infty}^4\|\nabla\nn_\ve\|_{H^{2s}}^2+\delta\|\nabla\vv_\ve\|_{H^{2s}}^2,\non\\
II_5\le&-\big\langle\nn_\ve\times(\Delta^{s+1}\nn_\ve\times\nn_\ve),\Delta^s\BOm_\ve\cdot\nn_\ve\big\rangle\non\\
&+C_\delta\|\nabla\nn_\ve\|_{L^\infty}^2\|\nn_\ve\|_{L^\infty}^4\|\nabla\nn_\ve\|_{H^{2s}}^2+\delta\|\nabla\vv_\ve\|_{H^{2s}}^2,\non
\end{align}
and by (\ref{dissipate relation}),
\begin{align}
II_3\le&-\Big\langle\big(\beta_1(\nn_\ve\nn_\ve:\Delta^s\DD_\ve)
\nn_\ve\nn_\ve+\beta_2|\nn_\ve|^4\Delta^s\DD_\ve
+\frac{\beta_3}{2}|\nn_\ve|^2(\nn_\ve\Delta^s\DD_\ve\cdot\nn_\ve
+\Delta^s\DD_\ve\cdot\nn_\ve\nn_\ve)\big),\Delta^s\DD_\ve\Big\rangle\non\\
&+C_\delta\|\nabla\nn_\ve\|_{L^\infty}^2\|\nn_\ve\|_{L^\infty}^6\|\vv_\ve\|_{H^{2s}}^2
+C_\delta\|\vv_\ve\|_{L^\infty}^2\|\nn_\ve\|_{L^\infty}^6\|\na\nn_\ve\|_{H^{2s}}^2
+\delta\|\nabla\vv_\ve\|_{H^{2s}}^2\non\\
\le& C_\delta\|\nn_\ve\|_{L^\infty}^6 \big(\|\nabla\nn_\ve\|_{L^\infty}^2+\|\vv_\ve\|_{L^\infty}^2\big)
\big(\|\vv_\ve\|_{H^{2s}}^2+\|\na\nn_\ve\|_{H^{2s}}^2\big)+\delta\|\nabla\vv_\ve\|_{H^{2s}}^2.\non
\end{align}

Summing up, we conclude that
\begin{align}\label{eq:energy-vf}
&\frac{1}{2}\frac{\ud}{\ud{t}}\langle\Delta^s\vv_\ve,\Delta^s\vv_\ve\rangle+\nu\langle\nabla\Delta^s\vv_\ve,\nabla\Delta^s\vv_\ve\rangle\non\\
&\le\mu_2\big\langle\nn_\ve\times(\Delta^{s+1}\nn_\ve\times\nn_\ve),\Delta^s\DD_\ve\cdot\nn_\ve\big\rangle
-\big\langle\nn_\ve\times(\Delta^{s+1}\nn_\ve\times\nn_\ve),\Delta^s\BOm_\ve\cdot\nn_\ve\big\rangle\non\\
&\quad+C_\delta\big(1+\|\nn_\ve\|_{L^\infty}^6\big)\big(\|\nabla\nn_\ve\|_{L^\infty}^2+\|\na\vv_\ve\|_{L^\infty}+\|\vv_\ve\|_{L^\infty}^2\big)
E_s(\vv_\ve,\nn_\ve)+\delta\|\nabla\vv_\ve\|_{H^{2s}}^2.
\end{align}

Summing up (\ref{eq:n-L2}), (\ref{eq:energy law-app}), (\ref{eq:energy-nf}) and (\ref{eq:energy-vf}),
then taking $\delta$ small enough, we get
\begin{align}\label{eq:energy}
&\frac{1}{2}\frac{\ud}{\ud{t}}E_s(\vv_\ve,\nn_\ve)+\f\nu 2\langle\nabla\vv_\ve,\nabla\vv_\ve\rangle+\f\nu 2\langle\nabla\Delta^s\vv_\ve,\nabla\Delta^s\vv_\ve\rangle\non\\
&\le C\big(1+\|\na\nn_0\|_{L^\infty}+\|\nn_\ve\|_{L^\infty}^6\big)\big(1+\|\nabla\nn_\ve\|_{L^\infty}^2+\|\na\vv_\ve\|_{L^\infty}+\|\vv_\ve\|_{L^\infty}^2\big)
E_s(\vv_\ve,\nn_\ve).
\end{align}

{\bf Step 3.} Existence of the solution

As $s\ge 2$, we deduced from Sobolev embedding and (\ref{eq:energy}) that
\begin{align}
\frac{\ud}{\ud{t}}E_s(\vv_\ve,\nn_\ve)+\nu\langle\nabla\vv_\ve,\nabla\vv_\ve\rangle+\nu\langle\nabla\Delta^s\vv_\ve,\nabla\Delta^s\vv_\ve\rangle
\le {\mathcal F}(E_s(\vv_\ve,\nn_\ve)),\non
\end{align}
where ${\mathcal F}$ is an increasing function with ${\mathcal F}(0)=0$. This implies that there exists $T>0$ depending only on $E_s(\vv_0,\nn_0)$ such that
for any $t\in [0,\min(T,T_\ve)]$,
\beno
E_s(\vv_\ve,\nn_\ve)+\nu\langle\nabla\vv_\ve,\nabla\vv_\ve\rangle+\nu\langle\nabla\Delta^s\vv_\ve,\nabla\Delta^s\vv_\ve\rangle
\le 2E_s(\vv_0,\nn_0),
\eeno
which in turn ensures that $T_\ve\ge T$ by a continuous argument. Thus, we obtain an uniform estimate for the approximate solution on $[0,T]$. Then
the existence of the solution can be deduced by a standard compactness argument. \vspace{0.1cm}

{\bf Step 4.} Uniqueness of the solution

Let $(\vv_1,\nn_1)$ and $(\vv_1,\nn_1)$ be two solutions of the system  (\ref{eq:EL}) with the same initial data. We denote
\beno
\delta_\vv=\vv_1-\vv_2,\quad \delta_\nn=\nn_1-\nn_2,\quad \delta_\hh=\hh_1-\hh_2,\quad \delta_\DD=\DD_1-\DD_2,\quad \delta_\BOm=\BOm_1-\BOm_2.
\eeno
Then $(\delta_\vv, \delta_\nn)$ satisfies
\begin{align*}
&\frac{\pa{\delta_\vv}}{\pa{t}}+\vv_1\cdot\nabla\delta_\vv+\delta_\vv\cdot\nabla\vv_2=-\nabla{p}+\nu\Delta\delta_\vv
+\nabla\cdot\big(\sigma_1(\vv_1,\nn_1)-\sigma_1(\vv_2,\nn_2)\\
&\qquad\qquad+\sigma_2(\nn_1)-\sigma(\nn_2)+\sigma^E(\nn_1)-\sigma^E(\nn_2)\big),\\
&\frac{\pa\delta_\nn}{\pa{t}}+\vv_1\cdot\nabla\delta_\nn+\delta_\vv\cdot\nabla\nn_2
=-\nn_1\times\big((\BOm_1\cdot\nn_1-\mu_1\hh_1-\mu_2\DD_1\cdot\nn_1)\times\nn_1\big)\\
&\qquad\qquad+\nn_2\times\big((\BOm_2\cdot\nn_2-\mu_1\hh_2-\mu_2\DD_2\cdot\nn_2)\times\nn_2\big).
\end{align*}
We make $L^2$ energy estimate for $\delta_\vv$ to get
\begin{align}
&\frac{1}{2}\frac{\ud}{\ud{t}}\|\delta_\vv\|_{L^2}^2+\nu\|\na\delta_\vv\|_{L^2}^2
=-\big\langle\delta_\vv\cdot\na\vv_2,\delta_\vv\big\rangle+
\big\langle\nabla\cdot(\sigma_1(\vv_1,\nn_1)-\sigma_1(\vv_2,\nn_2)),\delta_\vv\big\rangle\non\\
&\quad+\big\langle\nabla\cdot\big(\sigma_2(\nn_1)-\sigma_2(\nn_2)\big),\delta_\vv\big\rangle+
\big\langle\nabla\cdot\big(\sigma^E(\nn_1)-\sigma^E(\nn_2)\big),\delta_\vv\big\rangle\non\\
&=R_1+R_2+R_3+R_4,\non
\end{align}
and make $H^1$ energy estimate for $\delta_\nn$ to get
\begin{align}
\frac{1}{2}\frac{\ud}{\ud{t}}\|\na\delta_\nn\|_{L^2}^2
=&\big\langle\vv_1\cdot\nabla\delta_\nn+\delta_\vv\cdot\nabla\nn_2,\Delta\delta_\nn\big\rangle\non\\
&+\big\langle\nn_1\times\big((\BOm_1\cdot\nn_1-\mu_1\hh_1-\mu_2\DD_1\cdot\nn_1)\times\nn_1\big)\non\\
&\qquad-\nn_2\times\big((\BOm_2\cdot\nn_2-\mu_1\hh_2-\mu_2\DD_2\cdot\nn_2)\times\nn_2\big),
\Delta\delta_\nn\big\rangle\non\\
=&S_1+S_2.\non
\end{align}
Now we estimate $R_1,\cdots, R_4$. It is easy to see that
\beno
&&R_1\le \|\na \vv_2\|_{L^\infty}\|\delta_\vv\|_{L^2}^2,\\
&&R_4\le C\big(\|\nabla\nn_1\|_{L^\infty}+\|\nabla\nn_2\|_{L^\infty}\big)
\|\nabla\delta_\nn\|_{L^2}\|\nabla\delta_\vv\|_{L^2}.
\eeno
By (\ref{dissipate relation}), we have
\begin{align}
R_2&=-\big\langle\sigma_1(\delta_\vv,\nn_1),\nabla\delta_\vv\big\rangle
-\big\langle\sigma_1(\vv_2,\nn_1)-\sigma_1(\vv_2,\nn_2),\nabla\delta_\vv\big\rangle\non\\
&\le{C}\|\nabla\vv_2\|_{L^3}\|\na\delta_\nn\|_{L^2}\|\nabla\delta_\vv\|_{L^2}.\non
\end{align}
For $R_3$, we have
\begin{align*}
R_3=&\mu_2\big\langle\nn_1\times(\hh_1\times\nn_1)\nn_1-\nn_2\times(\hh_2\times\nn_2)\nn_2,\delta_\DD\big\rangle\\
&\quad+\big\langle\nn_1\times(\hh_1\times\nn_1)\nn_1-\nn_2\times(\hh_2\times\nn_2)\nn_2,\delta_\BOm\big\rangle\\
&=\mu_2\big\langle\nn_1\times(\delta_\hh\times\nn_1)\nn_1,\delta_\DD\big\rangle
+\big\langle\nn_1\times(\delta_\hh\times\nn_1)\nn_1,\delta_\BOm\rangle\\
&\quad+\mu_2\big\langle\nn_1\times(\hh_2\times\nn_1)\nn_1-\nn_2\times(\hh_2\times\nn_2)\nn_2,\delta_\DD\big\rangle\\
&\quad+\big\langle\nn_1\times(\hh_2\times\nn_1)\nn_1-\nn_2\times(\hh_2\times\nn_2)\nn_2,\delta_\BOm\big\rangle\\
&\le\mu_2\big\langle\nn_1\times(\delta_\hh\times\nn_1)\nn_1,\delta_\DD\rangle
+\big\langle\nn_1\times(\delta_\hh\times\nn_1)\nn_1,\delta_\BOm\big\rangle\\
&\quad+C\|\Delta\nn_2\|_{L^3}\|\na\delta_\nn\|_{L^2}\|\nabla\delta_\vv\|_{L^2}.
\end{align*}
Let us turn to estimate $S_1$ and $S_2$. It is easy to see that
\beno
S_1\le \|\na\vv_1\|_{L^\infty}\|\na\delta_\nn\|_{L^2}^2+
C\big(\|\na\nn_2\|_{L^\infty}+\|\Delta\nn_2\|_{L^3}\big)\|\na\delta_\vv\|_{L^2}\|\na\delta_\nn\|_{L^2},
\eeno
and for $S_2$, we have
\begin{align*}
S_2=&\big\langle\nn_1\times\big((\delta_\BOm\cdot\nn_1-\mu_1\delta_\hh
-\mu_2\delta_\DD\cdot\nn_1)\times\nn_1\big),\Delta\delta_\nn\big\rangle\\
&+\big\langle\nn_1\times\big((\BOm_2\cdot\nn_1-\mu_1\hh_2-\mu_2\DD_2\cdot\nn_1)\times\nn_1\big)\\
&\qquad-\nn_2\times\big((\BOm_2\cdot\nn_2-\mu_1\hh_2-\mu_2\DD_2\cdot\nn_2)\times\nn_2\big),\Delta\delta_\nn\big\rangle\\
\le&\big\langle\nn_1\times\big((\delta_\BOm\cdot\nn_1-\mu_1\delta_\hh
-\mu_2\delta_\DD\cdot\nn_1)\times\nn_1\big),\Delta\delta_\nn\big\rangle\\
&+C\big(\|\na \vv_2\|_{L^\infty}+\|\Delta\vv_2\|_{L^3}+\|\Delta\nn_2\|_{L^2}+\|\na\Delta\nn_2\|_{L^3}\big)\|\nabla\delta_\nn\|_{L^2}^2.
\end{align*}

Summing up all the above estimates, we obtain
\beno
\frac{\ud}{\ud{t}}\big(\|\delta_\vv\|_{L^2}^2+\|\na\delta_\nn\|_{L^2}^2\big)
\le C\big(\|\delta_\vv\|_{L^2}^2+\|\na\delta_\nn\|_{L^2}^2\big),
\eeno
which implies that $\delta_\vv(t)=0$ and $\delta_\nn(t)=0$ on $[0,T]$.\vspace{0.1cm}

{\bf Step 5.} Blow-up criterion

First of all, the solution of (\ref{eq:EL}) satisfies $|\nn|=1$ if $|\nn_0|=1$.
Thus, it holds that
\ben\label{eq:harmonic}
\nn\times(\Delta\nn\times\nn)=\Delta\nn+|\na\nn|^2\nn.
\een
Hence, $I_4$ in (\ref{eq:energy-n}) can be written as
\begin{align}
I_4=&-\mu_1\big\langle\Delta^{s+1}\nn+\Delta^s(|\na\nn|^2\nn), \Delta^{s+1}\nn\big\rangle\non\\
=&-\mu_1\big\langle\Delta^{s+1}\nn, \Delta^{s+1}\nn\big\rangle
+\mu_1\big\langle\Delta^s\big(\na(|\na\nn|^2)\nn\big)+\Delta^s(|\na\nn|^2\na\nn), \na\Delta^{s}\nn\big\rangle,\non
\end{align}
which along with Lemma \ref{lem:product} gives
\begin{align*}
I_4\le& -\mu_1\big\langle\Delta^{s+1}\nn, \Delta^{s+1}\nn\big\rangle
+C\|\na\nn\|_{L^\infty}\|\na\nn\|_{H^{2s}}\|\Delta^{s+1}\nn\|_{L^2}\\
&\quad+C\|\na\nn\|_{L^\infty}^2\|\na\nn\|_{H^{2s}}^2.
\end{align*}
On the other hand, we can bound $II_3$ as
\beno
II_3\le C\big(\|\na\nn\|_{L^\infty}^2+\|\na\vv\|_{L^\infty}\big)
\big(\|\na\nn\|_{H^{2s}}^2+\|\vv\|_{H^{2s}}^2\big)+\delta\|\na\Delta^s\vv\|_{L^2}^2
\eeno
by using the commutator estimate like
\beno
\|\na[\Delta^s, f]\na g\|_{L^2}\le C\big(\|\Delta^s\na f\|_{L^2}\|\na g\|_{L^\infty}+\|\na f\|_{L^\infty}\|\Delta^s\na g\|_{L^2}\big).
\eeno
From the proof in Step 2, we can deduce that
\begin{align}
\frac{\ud}{\ud{t}}E_s(\vv,\nn)
\le C\big(1+\|\nabla\nn\|_{L^\infty}^2+\|\na\vv\|_{L^\infty}\big)E_s(\vv,\nn).\non
\end{align}
Recall the following Logarithmic Sobolev inequality from\cite{BKM}:
\beno
\|\na \vv\|_{L^\infty}\le C\big(1+\|\na\vv\|_{L^2}+\|\na\times\vv\|_{L^\infty})\log(2+\|\vv\|_{H^k})
\eeno
for any $k\ge 3$. Thus, we have
\begin{align}
\frac{\ud}{\ud{t}}E_s(\vv,\nn)
\le C\big(1+\|\na\vv\|_{L^2}+\|\nabla\nn\|_{L^\infty}^2+\|\na\times\vv\|_{L^\infty}\big)\log\big(2+E_s(\vv,\nn)\big)E_s(\vv,\nn).
\non
\end{align}
Applying Gronwall's inequality twice, we infer that
\beno
E_s(\vv,\nn)\le E_s(\vv_0,\nn_0)
\exp\exp\Big(C\int_0^t\big(1+\|\na\vv\|_{L^2}+\|\nabla\nn\|_{L^\infty}^2+\|\na\times\vv\|_{L^\infty}\big)\ud \tau\Big)
\eeno
for any $t\in [0,T^*)$. Especially, if $T^*<+\infty$ and
\beno
\int_0^{T^*}\big(\|\nabla\nn\|_{L^\infty}^2+\|\na\times\vv\|_{L^\infty}\big)\ud t<+\infty,
\eeno
then $E_s(\vv,\nn)(t)\le C$ for any $t\in [0,T^*)$. Thus, the solution can
be extended after $t=T^*$, which contradicts the definition of $T^*$. The blow-up criterion follows.

\section{Global well-posedness for small initial data}

This section is devoted to the proof of Theorem \ref{thm:global}. Assume that
$(\vv,\nn)$ is the solution of the system (\ref{eq:EL}) on $[0,T]$ obtained in
Theorem \ref{thm:local}. We define
\begin{align}
&E_s(\vv,\nn)\eqdefa \|\nabla\nn\|_{L^2}^2+\|\nabla\Delta^s\nn\|_{L^2}^2
+\|\vv\|_{L^2}^2+\|\Delta^s\vv\|_{L^2}^2,\non\\
&D_s(\vv,\nn)\eqdefa   \mu_1\|\Delta\nn\|_{L^2}^2+\mu_1\|\Delta^{s+1}\nn\|_{L^2}^2
+\nu\|\na\vv\|_{L^2}^2+\nu\|\Delta^s\na\vv\|_{L^2}^2.\non
\end{align}
By the interpolation, there exist $c_0>0$ and $C_0>0$ such that
\beno
&&c_0\big(\|\na\nn\|_{H^{2s}}^2+\|\vv\|_{H^{2s}}^2\big)\le E_s(\vv,\nn)
\le C_0\big(\|\na\nn\|_{H^{2s}}^2+\|\vv\|_{H^{2s}}^2\big), \\
&&c_0\big(\|\Delta\nn\|_{H^{2s}}^2+\|\na\vv\|_{H^{2s}}^2\big)\le D_s(\vv,\nn)
\le C_0\big(\|\Delta\nn\|_{H^{2s}}^2+\|\na\vv\|_{H^{2s}}^2\big).
\eeno

The basic energy-dissipation law tells us that
\begin{align}
\frac{1}{2}\frac{\ud}{\ud{t}}\int_{\BR}|\vv|^2+|\nabla\nn|^2\ud\xx
+\int_{\BR}\big(\nu|\nabla\vv|^2+\mu_1|\nn\times\hh|^2\big)\ud\xx\le 0,\non
\end{align}
which along with (\ref{eq:harmonic}) implies that
\begin{align}
&\frac{1}{2}\frac{\ud}{\ud{t}}\int_{\BR}|\vv|^2+|\nabla\nn|^2\ud\xx
+\int_{\BR}\big(\nu|\nabla\vv|^2+\mu_1|\Delta\nn|^2\big)\ud\xx\non\\
&\le \mu_1\int_{\BR}|\na\nn|^4\ud\xx\le C\|\na\nn\|_{L^2}\|\Delta\nn\|_{L^2}^3
\le CE_s(\vv,\nn)D_s(\vv,\nn).\label{eq:energy-L2}
\end{align}
Similar to (\ref{eq:energy-n}), we have
\begin{align}\label{eq:energy-gn}
\frac{1}{2}\frac{\ud}{\ud{t}}\big\langle\nabla\Delta^s\nn, \nabla\Delta^s\nn\big\rangle
&=\big\langle\Delta^s(\vv\cdot\nabla\nn),\Delta^{s+1}\nn\big\rangle
+\big\langle\Delta^s\big[\nn\times\big((\BOm\cdot\nn)\times\nn\big)\big],\Delta^{s+1}\nn\big\rangle\nonumber\\
&\quad-\mu_2\big\langle\Delta^s\big[\nn\times\big((\DD\cdot\nn)\times\nn\big)\big],\Delta^{s+1}\nn\big\rangle
-\mu_1\big\langle\Delta^s\big[\nn\times\big(\Delta\nn\times\nn\big)\big],\Delta^{s+1}\nn\big\rangle\nonumber\\
&=I_1+I_2+I_3+I_4.
\end{align}
We get by Lemma \ref{lem:product} and Sobolev embedding that
\begin{align}
I_1\le& \|\Delta^{s}(\vv\cdot\na\vv)\|_{L^2}\|\Delta^{s+1}\nn\|_{L^2}\non\\
\le& C\|\vv\|_{L^\infty}\|\Delta^s\na \vv\|_{L^2}\|\Delta^{s+1}\nn\|_{L^2}
\le CE_s(\vv,\nn)^\f12D_s(\vv,\nn);
\end{align}
and
\begin{align}
I_2=&\big\langle\nn\times\big((\Delta^s\BOm\cdot\nn)\times\nn\big),\Delta^{s+1}\nn\big\rangle\non\\
&+\big\langle\Delta^s\big[\nn\times\big((\BOm\cdot\nn)\times\nn\big)\big],\Delta^{s+1}\nn\big\rangle
-\big\langle\nn\times\big((\Delta^s\BOm\cdot\nn)\times\nn\big),\Delta^{s+1}\nn\big\rangle\non\\
\le&\big\langle\nn\times\big((\Delta^s\BOm\cdot\nn)\times\nn\big),\Delta^{s+1}\nn\big\rangle\non\\
&+C\big(\|\na\nn\|_{L^\infty}\|\na\vv\|_{H^{2s-1}}+\|\Delta^s\nn\|_{L^2}\|\na\vv\|_{L^\infty}\big)
\|\Delta^{s+1}\nn\|_{L^2}\non\\
\le&\big\langle\nn\times\big((\Delta^s\BOm\cdot\nn)\times\nn\big),\Delta^{s+1}\nn\big\rangle
+CE_s(\vv,\nn)^\f12D_s(\vv,\nn);\\
I_3\le&-\mu_2\big\langle\nn\times\big((\Delta^s\DD\cdot\nn)\times\nn\big),\Delta^{s+1}\nn\big\rangle
+CE_s(\vv,\nn)^\f12D_s(\vv,\nn);
\end{align}
Similar to Step 4 in Section 4, we have
\begin{align}\label{eq:energy-gn4}
I_4\le& -\mu_1\big\langle\Delta^{s+1}\nn, \Delta^{s+1}\nn\big\rangle
+C\big(\|\na\nn\|_{L^\infty}+\|\na\nn\|_{L^\infty}^2\big)\|\Delta^s\nn\|_{H^{1}}\|\Delta^{s+1}\nn\|_{L^2}\non\\
\le&-\mu_1\big\langle\Delta^{s+1}\nn, \Delta^{s+1}\nn\big\rangle+C\big(E_s(\vv,\nn)^\f12+E_s(\vv,\nn)\big)D_s(\vv,\nn).
\end{align}

Summing up (\ref{eq:energy-gn})-(\ref{eq:energy-gn4}), we obtain
\begin{align}\label{eq:energy-gnf}
&\frac{1}{2}\frac{\ud}{\ud{t}}\big\langle\nabla\Delta^s\nn, \nabla\Delta^s\nn\big\rangle
+\mu_1\big\langle\Delta^{s+1}\nn, \Delta^{s+1}\nn\big\rangle\non\\
&\le\big\langle\nn\times\big((\Delta^s\BOm\cdot\nn)\times\nn\big),\Delta^{s+1}\nn\big\rangle-
\mu_2\big\langle\nn\times\big((\Delta^s\DD\cdot\nn)\times\nn\big),\Delta^{s+1}\nn\big\rangle\non\\
&\qquad+C\big(E_s(\vv,\nn)^\f12+E_s(\vv,\nn)\big)D_s(\vv,\nn).
\end{align}

Now we consider the estimate for the velocity. By Step 2 in Section 4, we have
\begin{align}\label{eq:energy-gv}
&\frac{1}{2}\frac{\ud}{\ud{t}}\langle\Delta^s\vv,\Delta^s\vv\rangle+\nu\langle\nabla\Delta^s\vv,\nabla\Delta^s\vv\rangle\non\\
&=-\langle\Delta^s(\vv\cdot\nabla\vv),\Delta^s\vv\rangle
+\langle\Delta^s(\nabla\nn\odot\nabla\nn), \Delta^s\nabla\vv\rangle\non\\
&\quad-\Big\langle\Delta^s\big(\beta_1(\nn\nn:\DD)\nn\nn+\beta_2\DD
+\frac{\beta_3}{2}(\nn\DD\cdot\nn+\DD\cdot\nn\nn)\big),\Delta^s\DD\Big\rangle\non\\
&\qquad+\mu_2\big\langle\Delta^s\big(\nn\times(\hh\times\nn)\nn\big),\Delta^s\DD\big\rangle
-\big\langle\Delta^s(\nn\times(\hh\times\nn)\nn),\Delta^s\BOm\big\rangle\non\\
&=II_1+II_2+II_3+II_4+II_5.
\end{align}
We get by Lemma \ref{lem:product} and Sobolev embedding that
\begin{align}
II_1&\le C\|\vv\|_{L^\infty}\|\Delta^s\vv\|_{L^2}\|\Delta^s\na\vv\|_{L^2}\le CE_s(\vv,\nn)^\f12D_s(\vv,\nn);\\
II_2&\le C\|\na\nn\|_{L^\infty}\|\Delta^s\na\nn\|_{L^2}\|\Delta^s\na\vv\|_{L^2}\le CE_s(\vv,\nn)^\f12D_s(\vv,\nn);
\end{align}
and by Proposition \ref{prop:Leslie condition},
\begin{align}
II_3\le&-\Big\langle\big(\beta_1(\nn\nn:\Delta^s\DD)
\nn\nn+\beta_2\Delta^s\DD+\frac{\beta_3}{2}(\nn\Delta^s\DD\cdot\nn
+\Delta^s\DD\cdot\nn\nn)\big),\Delta^s\DD\Big\rangle\non\\
&+C\big(\|\na\vv\|_{L^\infty}\|\Delta^s\nn\|_{L^2}+\|\na\nn\|_{L^\infty}\|\na\vv\|_{H^{2s-1}}\big)
\|\Delta^s\na\vv\|_{L^2}\non\\
\le& CE_s(\vv,\nn)^\f12D_s(\vv,\nn);
\end{align}
Similarly, we have
\begin{align}\label{eq:energy-gv5}
II_4+II_5\le&\mu_2\big\langle\nn\times(\Delta^{s+1}\nn\times\nn),\Delta^s\DD\cdot\nn\big\rangle
-\big\langle\nn\times(\Delta^{s+1}\times\nn),\Delta^s\BOm\cdot\nn\big\rangle\non\\
&+CE_s(\vv,\nn)^\f12D_s(\vv,\nn).
\end{align}

Summing up (\ref{eq:energy-gv})-(\ref{eq:energy-gv5}), we obtain
\begin{align}\label{eq:energy-gvf}
&\frac{1}{2}\frac{\ud}{\ud{t}}\langle\Delta^s\vv,\Delta^s\vv\rangle+\nu\langle\nabla\Delta^s\vv,\nabla\Delta^s\vv\rangle\non\\
&\le\mu_2\big\langle\nn\times(\Delta^{s+1}\nn\times\nn),\Delta^s\DD\cdot\nn\big\rangle
-\big\langle\nn\times(\Delta^{s+1}\times\nn),\Delta^s\BOm\cdot\nn\big\rangle\non\\
&\qquad+CE_s(\vv,\nn)^\f12D_s(\vv,\nn).
\end{align}

It follows from (\ref{eq:energy-L2}), (\ref{eq:energy-gnf}) and (\ref{eq:energy-gvf}) that
\begin{align}
\frac{1}{2}\frac{\ud}{\ud{t}}E_s(\vv,\nn)+D_s(\vv,\nn)\le C\big(E_s(\vv,\nn)^\f12+E_s(\vv,\nn)\big)D_s(\vv,\nn).\non
\end{align}
This implies that there exists an $\ve_0>0$ such that
if $E_s(\vv_0,\nn_0)\le \ve_0$, then
\beno
E_s(\vv,\nn)(t)\le E_s(\vv_0,\nn_0)\quad \textrm{for any }t\in [0,T].
\eeno
Thus, the solution is global in time by blow-up criterion in Theorem \ref{thm:local}.

\section*{Acknowledgements}

{\small The authors thank Professor Fang-Hua Lin for helpful discussions and suggestions.
P. Zhang is partly supported by NSF of China under Grant 11011130029.
Z. Zhang is partly supported by NSF of China under Grant 10990013 and 11071007.}


\begin{thebibliography}{50}

\bibitem{Bah} H. Bahouri, J.-Y. Chemin and R. Danchin,
{\it Fourier analysis and nonlinear partial differential equations},
Fundamental Principles of Mathematical Sciences, 343, Springer, Heidelberg, 2011.

\bibitem{BKM} J. T. Beale, T. Kato and A. Majda, {\it Remarks on the breakdown of smooth solutions for
the 3-D Euler equation}, Comm. Math. Phys., 94(1984), 61-66.

\bibitem{EZ} W. E and P. Zhang,
{\it A Molecular Kinetic Theory of Inhomogeneous Liquid Crystal Flow
and the Small Deborah Number Limit}, Methods and Appications of
Analysis, {\bf13}(2006), 181-198.

\bibitem{E-61} J. Ericksen, {\it Conservation laws for liquid crystals},
Trans. Soc. Rheol. , {5}(1961), 22-34.

\bibitem{E-91} J. Ericksen, {\it Liquid crystals with variable degree of orientation}, Arch. Rat. Mech. Anal., 113
(1991), 97-120.

\bibitem{Hong} M. Hong, {\it Global existence of solutions of the simplified Ericksen-Leslie system in dimension two},
Calc. Var. Partial Differential Equations, 40(2011), 15-36.

\bibitem{HW} T. Huang and C. Wang, {\it Blow up criterion for nematic liquid crystal flows}, Communications in Partial Differential Equations,
37(2012), 875-884.

\bibitem{KD} N. Kuzuu and M. Doi,
{\it Constitutive equation for nematic liquid crystals under weak
velocity gradient derived from a molecular kinetic equation,}
Journal of the Physical Society of Japan, 52(1983), 3486-3494.

\bibitem{Les} F. M. Leslie, {\it Some constitutive equations for liquid crystals}, Arch. Rat. Mech. Anal., 28
(1968), 265-283.


\bibitem{LL-CPAM}   F.-H. Lin and C. Liu, {\it Nonparabolic dissipative systems modeling the flow of liquid crystals},
Comm. Pure Appl. Math., 48(1995), 501-537.

\bibitem{LL-DCDS} F.-H. Lin and C. Liu, {\it Partial regularities of the nonlinear dissipative systems
modeling the fow of liquid crystals}, Disc. Conti. Dyna. Sys., 2 (1996), 1-23.

\bibitem{LL-ARMA} F.-H. Lin and C. Liu, {\it Existence of solutions for the Ericksen-Leslie system},
Arch. Ration. Mech. Anal., 154(2000), 135-156.



\bibitem{LLW} F.-H. Lin, J. Lin and C. Wang,
{\it Liquid crystal flows in two dimensions}, Arch. Ration. Mech. Anal.,
 197 (2010),  297-336.

 \bibitem{LW}  F.-H. Lin and C. Wang, {\it On the uniqueness of heat flow of harmonic maps and hydrodynamic flow of nematic liquid crystals},
 Chin. Ann. Math. Ser. B, 31(2010), 921-938.

 \bibitem{Parodi} O. Parodi, {\it Stress tensor for a nematic liquid crystal,}
Journal de Physique, 31 (1970), 581-584.



\bibitem{Wang} C. Wang, {\it Well-posedness for the heat flow of harmonic maps and the liquid crystal flow with rough initial data},
 Arch. Ration. Mech. Anal., 200(2011),  1-19.

\bibitem{WZZ} W. Wang, P. Zhang and Z. Zhang, {\it The small Deborah number limit of the Doi-Onsager equation  to the
Ericksen-Leslie equation}, arXiv:1206.5480.

\bibitem{WXL} H. Wu, X. Xu and C. Liu,  {\it On the general Ericksen Leslie system: Parodi's relation,
well-posedness and stability}, arXiv:1105.2180.

\bibitem{XZ} X. Xu and Z. Zhang, {\it Global regularity and uniqueness of weak solution for the 2-D liquid crystal flows},
 J. Differential Equations, 252 (2012),  1169-1181.


\end{thebibliography}
\end{document}